\documentclass[12pt]{amsart}

\usepackage{amsmath, amsthm, amsfonts, amssymb, graphicx, mathrsfs,
latexsym}

\makeatletter 
\newtheorem*{rep@theorem}{\rep@title}
\newcommand{\newreptheorem}[2]{%
\newenvironment{rep#1}[1]{%
 \def\rep@title{#2 \ref{##1}}%
 \begin{rep@theorem}}%
 {\end{rep@theorem}}}
\makeatother

	\newtheorem{dfn}{Definition}[section]
    \newtheorem{thm}[dfn]{Theorem}
    \newreptheorem{thm}{Theorem}
	\newtheorem{prop}[dfn]{Proposition}
	\newreptheorem{prop}{Proposition}
	
	\newtheorem{lem}[dfn]{Lemma}
	\newreptheorem{lem}{Lemma}

	\newtheorem{conj}[dfn]{Conjecture}

	\newcounter{yon}
	\numberwithin{equation}{section}

\newcommand{\del}{\partial}

\newcommand{\scal}{\mathop{\mathrm{scal}} \nolimits}
\DeclareMathOperator{\Vol}{Vol}
\DeclareMathOperator{\hyp}{hyp}
\DeclareMathOperator{\dist}{dist}

\DeclareMathOperator{\const}{const}
\DeclareMathOperator{\Area}{Area}
\newcommand{\co}{\colon\thinspace}

\begin{document}

\begin{abstract}
In this paper we prove the following. Let $\Sigma$ be an $n$--dimensional closed hyperbolic manifold and let $g$ be a Riemannian metric on $\Sigma \times \mathbb{S}^1$.  Given an upper bound on the volumes of unit balls in the Riemannian universal cover $(\widetilde{\Sigma\times \mathbb{S}^1},\widetilde{g})$, we get a lower bound on the area of the $\mathbb{Z}_2$--homology class $[\Sigma \times \ast]$ on $\Sigma \times \mathbb{S}^1$, proportional to the hyperbolic area of $\Sigma$.  The theorem is based on a theorem of Guth and is analogous to a theorem of Kronheimer and Mrowka involving scalar curvature.
 \end{abstract}

	\title[Macroscopic scalar curvature and areas of cycles]
	{Macroscopic scalar curvature and areas of cycles}

    \author[Hannah Alpert]{Hannah Alpert}
    \address{Ohio State University, 43210-1174 USA}
    \email{alpert.19@osu.edu}

    \author[Kei Funano]{Kei Funano}
	\address{Tohoku
	University, Sendai, 980-8579 Japan}
	\email{kfunano@tohoku.ac.jp}

	\subjclass[2010]{53C21, 53C23}
	\keywords{Macroscopic scalar curvature, minimal surface, hyperbolic manifold}
\date{\today}

\maketitle

\section{Introduction} 
Our result in the present paper is based on the following theorem of
Guth in \cite[Theorem 2]{Guth11}.  His theorem, which we refer to as the
volume theorem, relates the volume of a manifold to the
volumes of unit balls in the universal cover.  To see how these
quantities might be related, we can imagine starting with a closed
hyperbolic manifold and rescaling the metric.  As we scale it so the
manifold gets bigger, the curvature gets closer to zero so the unit
balls in the universal cover have volumes closer to the Euclidean unit
ball.  Scaling in the other direction, the manifold gets smaller, and
the unit balls in the cover get wiggly and large.  To state the theorem
we use the following notation.  For a Riemannian manifold $(M,g)$ and
radius $r$, let $V_{(M,g)}(r)$ denote the supremal volume of a ball of
radius $r$ in $(M,g)$. We denote by $(\widetilde{M},\widetilde{g})$ the
Riemannian universal cover of $(M,g)$. By $\const
(\alpha,\beta,\cdots)$ we mean some positive absolute constant depending only on $\alpha,\beta,\cdots$.

\begin{thm}[{Volume theorem, \cite{Guth11}}]\label{thm-vol}
 Let $(M, \hyp)$ be an $n$--dimensional closed hyperbolic manifold and let $g$ be another metric on $M$. Suppose that
\[V_{(\widetilde{M}, \widetilde{g})}(1) \leq V_{\mathbb{H}^n}(1),\]
where $\mathbb{H}^n$ is the $n$--dimensional hyperbolic
 space.  Then we have
\[\Vol(M, g) \geq \const(n) \cdot \Vol(M, \hyp).\]
\end{thm}

In this paper we consider the following question: Does a similar theorem
hold if we take the product of a hyperbolic manifold and a closed
manifold? If we have a metric on the product, for
which the unit balls in the universal cover are not too large, does this guarantee that every
copy of the hyperbolic manifold in the product is not too small? It turns out that the answer is affirmative if the second factor is a circle.  For the first factor, we use the letter $\Sigma$ to refer
to the hyperbolic manifold even though $\Sigma$ may have dimension greater than 2, because $\Sigma$ is a hypersurface in the product.  We also refer to volume in this dimension as area, and call
the new theorem the area theorem.  It is the main theorem of this paper.

\begin{thm}[Area theorem]\label{thm-homology}
  Let $(\Sigma, \hyp)$ be an $n$--dimensional closed hyperbolic manifold, and let $g$ be a metric on $M = \Sigma \times \mathbb{S}^1$ such that 
\[V_{(\widetilde{M}, \widetilde{g})}(1) \leq V_0,\]
for some value $V_0$.  Let $Z$ be a smooth embedded hypersurface of $M$ homologous to $\Sigma \times \ast$ under $\mathbb{Z}_2$ coefficients.  Then we have
\[\Area Z \geq \const(n, V_0) \cdot \Area (\Sigma, \hyp).\]
\end{thm}

In contrast the following proposition says that an analogous statement does not hold for the product of a hyperbolic manifold with another manifold of dimension greater than $1$.  In that case, conversation with Guth produced examples where one copy of the hyperbolic manifold is very small but none of the unit balls in the universal cover are large. We thank him to allow us to add these examples in this paper:

\begin{prop}\label{prop:codim2}
Let $\Sigma$ be a closed hyperbolic manifold of dimension $n$, and let $X$ be any closed manifold of dimension $m \geq 2$.  There exists a constant $V_0 = 2\cdot V_{\mathbb{R}^{n+m}}(1)$ such that the following is true.  For any $\varepsilon > 0$ there is a Riemannian metric $g$ on $\Sigma \times X$ such that the volumes of unit balls in the universal cover satisfy
\[V_{(\widetilde{\Sigma \times X}, \widetilde{g})}(1) < V_0,\]
but simultaneously there is a submanifold isotopic to $\Sigma \times \ast$ with $n$--dimensional volume less than $\varepsilon$.
\end{prop}

In the remainder of the introduction we mention some context about how the subject of this paper is related to the study of minimal surfaces and scalar curvature.  However, in order to understand the proof of the main theorem we do not require any knowledge of minimal surfaces and scalar curvature.

\subsection{Macroscopic scalar curvature} 

In the hypotheses of the volume theorem (Theorem~\ref{thm-vol}) and the area theorem (Theorem~\ref{thm-homology}), the bound on volumes of balls in the universal cover is a bound on the macroscopic scalar curvature of the manifold, defined by Guth in~\cite{Guth10'} based on ideas of Gromov.  In that essay, Guth thoroughly explains the progression of ideas into which his volume theorem and the present paper fit; we summarize a little of this history here.  Gromov was thinking of the fact that the scalar curvature of a manifold describes how the volumes of tiny balls compare to balls of the same radius in Euclidean space.  The difference between volumes has a leading term expressible in terms of the scalar curvature: we have the asymptotic expansion 
\begin{align*}
 \Vol B(p,r)=\omega_n
 r^{n}\Big(1-\frac{\scal_g(p)}{6(n+2)}r^2+O(r^3)\Big) \ \ \ (r\to 0),
 \end{align*}where $\omega_n$ is the volume of a unit ball
 in the $n$--dimensional Euclidean space (\cite[Theorem 3.98]{Gallot04}).  Understanding the relationship between scalar curvature and the topology of $M$ is a major problem in differential geometry. However since the scalar curvature says only local information it is often difficult to derive global information of the ambient manifold from it.
 
 The macroscopic scalar curvature is defined as follows.  Let $(M,g)$ be a Riemannian manifold. For a radius $r$ and $p\in M$ we
denote by $\widetilde{V}(p,r)$ the volume of the ball of radius $r$
centered at a lift of $p$. In any given dimension, for any real number $S$ there is a scaling of either hyperbolic space, Euclidean space, or the sphere that has scalar curvature $S$; we define $\widetilde{V}_S(r)$ as the volume of a ball of radius $r$ in that space. Then the \emph{macroscopic scalar curvature} at scale $r$ at $p$ is defined as the number $S$ so that
$\widetilde{V}(p,r)=\widetilde{V}_S(r)$. For example any flat torus has
the macroscopic scalar curvature zero at any scale. Note that if we do
not take the universal cover in the definition, then a flat torus might
have positive `macroscopic scalar curvature' at a certain scale.

 In view of macroscopic scalar curvature Guth's volume theorem (Theorem \ref{thm-vol}) corresponds to the
 following Schoen conjecture for
 scalar curvature up to a multiplicative constant factor.  The conjecture remains open in high dimensions.
 \begin{conj}[{\cite{Schoen89}}]If $(M,\hyp)$ is an $n$--dimensional
  closed hyperbolic manifold and $g$ is another metric on $M$ with
  $\scal_g\geq \scal_{\hyp}$, then $\Vol (M,g)\geq \Vol (M,\hyp)$.
  \end{conj}

 The area theorem (Theorem \ref{thm-homology}) is a counterpart of the following result due to
 Kronheimer and Mrowka (for integer coefficients): 
\begin{prop}[{\cite{Kronheimer97}}]\label{thm-KM}Let $\Sigma$ be a
 closed orientable surface of genus $\geq 2$ and let $g$ be a Riemannian
 metric on $M = \Sigma \times \mathbb{S}^1$. Assume that $\scal_g\geq
 \scal_{g_0}$, where $g_0$ is the product metric of $g_{\hyp}$ on $\Sigma$ and the
 standard metric on $\mathbb{S}^1$. Let $Z$ be a surface in $M$ homologous to $\Sigma \times \ast$ under integer coefficients.  Then we have
 \begin{align*}
  \Area Z \geq  \Area(\Sigma,\hyp).
  \end{align*}
 \end{prop}
The above proposition is a special case of a result in
\cite[Section 3]{Kronheimer97}. Kronheimer and Mrowka proved the result
by using the Seiberg-Witten monopole equations. They mentioned that the
result can be obtained by using the existence of a stable minimal
surface and the second variation inequality. In the appendix we demonstrate
the detailed proof of Proposition \ref{thm-KM} using this second method. An
analogous result of Proposition \ref{thm-KM} (for scalar curvature) seems unknown for higher
dimension. A crucial point of Theorem \ref{thm-homology} is that it holds
for any dimension. 

In Section~\ref{sect:outline} we outline the proof of the area theorem (Theorem~\ref{thm-homology}) by stating two lemmas and showing how they imply the area theorem.  In Section~\ref{section-proofs} we prove the first lemma, the modified stability lemma (Lemma~\ref{modified-stability}).  In Section~\ref{sect:proof} we prove the second lemma, the modified volume theorem (Lemma~\ref{modified-volume}), completing the proof of the area theorem.  In Section~\ref{sect:codim 2} we prove Proposition~\ref{prop:codim2} about the case of codimension greater than $1$.  And, the appendix addresses Proposition~\ref{thm-KM} of Kronheimer and Mrowka.

\section{Proof of the area theorem}\label{sect:outline}

In this section we state the modified stability lemma (Lemma~\ref{modified-stability}) and the modified volume theorem (Lemma~\ref{modified-volume}), and we prove the area theorem (Theorem~\ref{thm-homology}) assuming these two lemmas.  The proofs of the lemmas appear in Section~\ref{section-proofs} and Section~\ref{sect:proof}.

The proof of the area theorem follows these steps:
\begin{enumerate}
\item Take an almost area-minimizing closed hypersurface $Z$ in $M = (\Sigma \times \mathbb{S}^1, g)$ homologous to $\Sigma \times \ast$ under $\mathbb{Z}_2$ coefficients.
\item Using the hypothesis about areas of balls in $\widetilde{M}$, show that the balls of radius~$\frac{1}{2}$ in $Z$ are not too large in area.
\item Apply a version of the volume theorem (Theorem~\ref{thm-vol}) to $Z$ to show that the total area of $Z$ is not too small compared to the hyperbolic area of $\Sigma$.
\end{enumerate}

For the first step, it would be convenient if we could really choose a minimal surface; then we could use properties such as monotonicity.  However, the existence of minimal surfaces has its subtleties, and if dimension is greater than 6 we cannot hope to take an area-minimizing closed hypersurface without singularities (see \cite{Simon83}).  Fortunately we do not need minimal surfaces for this proof.  We say that a hypersurface is \emph{minimizing up to $\delta$} if its area is within $\delta$ of the infimal area of a smooth embedded hypersurface in its homology class.  We choose $Z$ to be such a hypersurface, homologous to the cross-sections $\Sigma \times \ast$ of $M$. 

For the second step, we adapt the stability lemma from the paper~\cite{Guth10} of Guth.  Here is the statement of the modified stability lemma; its proof appears in Section~\ref{section-proofs} along with the original statement of the stability lemma.

\begin{lem}[Modified stability lemma]\label{modified-stability}
  Let $\Sigma$ be an $n$--dimensional closed manifold, and let $g$ be a Riemannian metric on $M = \Sigma \times \mathbb{S}^1$.  Suppose that $Z^n \subset M$ is a smooth embedded hypersurface, homologous to $\Sigma \times \ast$ under $\mathbb{Z}_2$ coefficients, and minimizing up to $\delta$ among all such embedded hypersurfaces.  Let $\widehat{M}$ denote the covering space of $M$ diffeomorphic to $\Sigma \times \mathbb{R}$.  Let $p$ be any point in $M$, and suppose that the unit ball in $\widehat{M}$ centered at a lift $\widehat{p}$ of $p$ has volume at most $V_0$.  Then the following inequality holds:
\[\Area(Z \cap B(p, 1/2)) \leq 2V_0 + \delta.\]
\end{lem}

For the third step, we use the following version of the volume theorem; its proof appears in Section~\ref{sect:proof}.

 \begin{lem}[Modified volume theorem]\label{modified-volume}
 Let $\Sigma$ be an $n$--dimensional closed hyperbolic manifold, and let $(Z, g)$ be an $n$--dimensional closed Riemannian manifold.  Suppose that $f \co (Z, g) \rightarrow \Sigma$ is a degree $1$ map such that $f(\gamma)$ is contractible for every loop $\gamma \subset Z$ of length less than $1$, and suppose that
\[V_{(Z, g)}(1/2) \leq V_0.\]
Then we have
\[\Vol (Z, g) \geq \const(n, V_0) \cdot \Vol(\Sigma, \hyp).\]
\end{lem}

Having stated the two lemmas, we are ready to show how they imply the area theorem.

\begin{proof}[Proof of the area theorem (Theorem~\ref{thm-homology})]
 First we take a finite cover of $M$ in the $\Sigma$ direction so that the only remaining loops of length at most $1$ are in the $\mathbb{S}^1$ direction.  Specifically, using compactness of $M$ we see that there are only finitely many classes in $\pi_1(M) = \pi_1(\Sigma) \times \pi_1(\mathbb{S}^1)$ with length at most $1$.  Because fundamental groups of hyperbolic manifolds are residually finite, we may find a finite-index normal subgroup $K$ of $\pi_1(\Sigma)$ that excludes all nonzero $\pi_1(\Sigma)$ components of these short loops.  Then $K \times \pi_1(\mathbb{S}^1)$ is a finite-index normal subgroup of $\pi_1(M)$.  The corresponding covering space $(\widehat{M}, \widehat{g})$ is diffeomorphic to $\widehat{\Sigma} \times \mathbb{S}^1$, where $\widehat{\Sigma}$ is the cover of $\Sigma$ corresponding to the normal subgroup $K$; let $k$ be the index of $K$ in $\pi_1(\Sigma)$, and of $\widehat{\Sigma}$ over $\Sigma$, and of $\widehat{M}$ over $M$.  In $\widehat{M}$, every loop of length at most $1$ projects to a null-homotopic loop in $\widehat{\Sigma}$.  The preimage of $Z$ in $\widehat{M}$ is some $k$-fold cover $\widehat{Z}$ of $Z$, and $\widehat{Z}$ is homologous to $\widehat{\Sigma} \times \ast$.  It suffices to prove the area theorem for $\widehat{Z}$, $\widehat{\Sigma}$, and $\widehat{M}$ instead of $Z$, $\Sigma$, and $M$, because $\Area \widehat{Z} = k \cdot \Area Z$ and $\Area (\widehat{\Sigma}, \hyp) = k \cdot \Area (\Sigma, \hyp)$.  Thus we may relabel our spaces and assume that $M$ has the property that every loop of length at most $1$ projects to a null-homotopic loop in $\Sigma$.

Note that we could take a further cover in the $\mathbb{S}^1$ direction so that every loop of length at most $1$ is null-homotopic.  However, if the cover has many sheets, then the preimage of $Z$ is homologous to a large multiple of $\Sigma$ and has area many times that of $Z$.  Therefore, a lower bound on the area of hypersurfaces homologous to $\Sigma$ in the cover does not easily imply any bound on the area of $Z$.  This fact that $Z$ does not behave well under many-sheeted covers in the $\mathbb{S}^1$ direction is our reason for proving the modified stability lemma (Lemma~\ref{modified-stability}) instead of being able to use Guth's original stability lemma (Theorem~\ref{stability-lemma}); more detail appears in Section~\ref{section-proofs}.

It suffices to prove the area theorem only for those hypersurfaces $Z$ that are minimizing up to $\delta$ (for, say, $\delta = 1$), because the other hypersurfaces have even greater area.  Thus, we apply the modified stability lemma (Lemma~\ref{modified-stability}) to $Z$, to obtain
\[V_{(Z, g\vert_Z)}(1/2) \leq 2V_0 + \delta,\]
and then apply the modified volume theorem (Lemma~\ref{modified-volume}), taking $f$ to be the composed inclusion and projection $Z \rightarrow M \rightarrow \Sigma$, to obtain the desired inequality
\[\Area Z \geq \const(n, V_0) \cdot \Area (\Sigma, \hyp).\]
 This completes the proof of the area theorem.
\end{proof}

\section{Proof of the modified stability lemma}\label{section-proofs}

In this section we prove the modified stability lemma.  This lemma is based on the stability lemma, which we reproduce here.  Given a cohomology class $\alpha$ in degree $1$, its \emph{length} $L(\alpha)$ is defined to be the infimal length of any $1$--cycle that has nonzero pairing with $\alpha$.

 \begin{thm}[\cite{Guth10}, Stability lemma]\label{stability-lemma}
 Let $(M, g)$ be an $(n+1)$--dimensional closed Riemannian manifold.  Let $\alpha \in H^1(M, \mathbb{Z}_2)$.  Suppose that $Z^n \subset M$ is a smooth embedded surface Poincar\'e dual to $\alpha$ and minimizing up to $\delta$ (among embedded surfaces in its homology class).  Suppose that $L(\alpha) > 2$.  Let $p$ be any point in $M$.  Then the following inequality holds:
\[\Area(Z \cap B(p, 1/2)) \leq 2\Vol (B(p, 1)) + \delta.\]
\end{thm}

The stability lemma gets its name by analogy with the stability inequality involving scalar curvature and minimal hypersurfaces, which we mention in the appendix.

The stability lemma is insufficient for our purposes in proving the area theorem (Theorem~\ref{thm-homology}): we would like to apply it while taking $\alpha$ to be the cohomology class on $M = \Sigma \times \mathbb{S}^1$ corresponding to the $\mathbb{S}^1$ factor, but we do not have any control over the length of $\alpha$.  The modified stability lemma removes the hypothesis about the length of $\alpha$.

In fact, we can generalize the modified stability lemma so that the stability lemma becomes a special case.  We take $(M, g)$, $\alpha$, and $Z$ as in the stability lemma.  Then we define $\widehat{M}$ to be the covering space of $M$ constructed as follows: take a classifying map $M \rightarrow \mathbb{S}^1$ so that $\alpha$ is the pullback of the fundamental class of $\mathbb{S}^1$, and then pull back the universal cover $\mathbb{R} \rightarrow S^1$ along this classifying map.  Then, if we define $V_0$ as in the modified stability lemma, the conclusion of the modified stability lemma holds.  The proof is the same as that of the modified stability lemma.  In the case $L(\alpha) > 2$, the unit balls in $\widehat{M}$ have the same shape as their images in $M$, and we recover the stability lemma.

The proof of the modified stability lemma closely follows the proof of the stability lemma but requires an additional idea.  Roughly, both proofs show that if the desired inequality does not hold, then we can replace the portion of $Z$ inside $B(p, 1/2)$ with something that is homologous but has significantly smaller area, contradicting the assumption that $Z$ is minimizing up to $\delta$.

\begin{proof}[Proof of modified stability lemma (Lemma~\ref{modified-stability})]
 As in the proof of the stability lemma in~\cite{Guth10}, we can use the
 coarea inequality to find $t$ between $1/2$ and $1$ such that the
 sphere in $\widehat{M}$ with radius $t$ and center $\widehat{p}$ has
 area at most $2V_0$.  We would like to be able to choose this $t$ such
 that the corresponding sphere is a smooth submanifold of $\widehat{M}$,
 and its image in $M$ is an immersed hypersurface with transverse
 self-intersections and transverse intersections with $Z$. This may be tricky. Instead we use Lemma~\ref{lem-transverse} (which appears after the present proof) to find a smooth function $f \co \widehat{M} \rightarrow \mathbb{R}$ that is $\varepsilon$--close to the function $\dist(-, \widehat{p})$, such that for a dense set of values of $t$, the level set $f^{-1}(t)$ has the desired transversality properties.  The $\varepsilon$--closeness guarantees that the sublevel set $f^{-1}[-\varepsilon, 1/2 + \varepsilon]$ contains the ball $B(\widehat{p}, 1/2)$ in $\widehat{M}$, and that the sublevel set $f^{-1}[-\varepsilon, 1 - \varepsilon]$ is contained in the ball $B(\widehat{p}, 1)$.  We use the coarea inequality to find $t$ between $1/2 + \varepsilon$ and $1-\varepsilon$ such that $f^{-1}(t)$ has the desired transversality properties and has area at most $[(1-\varepsilon) - (1/2 + \varepsilon)]^{-1}V_0$.

Let $\widehat{B}$ denote the region $f^{-1}[-\varepsilon, t]$ in $\widehat{M}$, with boundary $\del \widehat{B} = f^{-1}(t)$.  Let $B$ denote the image of $\widehat{B}$ in $M$, and let $S$ denote the image of $\del \widehat{B}$ in $M$, the immersed hypersurface.  Note that $\del B$ is part of $S$ but may not be all of $S$.

The proof of the modified stability lemma has the following structure, similar to the proof of the stability lemma.  We show that there is a smooth embedded hypersurface homologous to $Z$, obtained by replacing $Z \cap B$ by some subset of $S$ (and then smoothing).  Because $Z$ is minimizing up to $\delta$, we conclude
\[\Area (Z \cap B) \leq \Area S + \delta.\]
Then because $Z \cap B(p, 1/2) \subseteq Z \cap B$ and because $\Area S \leq [(1-\varepsilon) - (1/2 + \varepsilon)]^{-1}V_0$, by sending $\varepsilon$ to $0$ we obtain the desired inequality
\[\Area (Z \cap B(p, 1/2)) \leq 2V_0 + \delta.\]

In order to replace $Z \cap B$ by a subset of $S$, we claim that it suffices to find a subset $Z_1$ of $S$ in the interior of $B$ such that every loop in the interior of $B$ has the same intersection number (mod $2$) with $Z_1$ as with $Z$.  If we can do this, then we can color the regions of $B \setminus (Z \cup S)$ with two colors so that any path between regions of the same color has an even intersection number with $Z \cup Z_1$.  In particular, any two regions that border the same open subset of $Z$ must be opposite colors.  Thus, we can select one color and modify $Z$ by adding (mod $2$) the boundaries of all regions of that color.  The result is a cycle homologous to $Z$ in which all of $Z \cap B$ has been replaced by pieces of $S$.  As we add the boundaries of the regions one by one, the property of being smoothable to an embedded hypersurface is preserved, so the end result can be modified slightly to give a smooth embedded hypersurface.

Thus it suffices to find $Z_1 \subseteq S$ such that every loop in $B$ intersects $Z$ and $Z_1$ with the same parity.  It suffices to consider loops at $p$.  Every loop $\gamma$ at $p$ lifts in $\widehat{M}$ to a path $\widetilde{\gamma}$ that begins at $\widehat{p}$ and ends at another preimage $\widehat{p} + k$ of $p$, for some $k \in \mathbb{Z}$.  Because we have assumed that $Z$ is homologous to $\Sigma \times \ast$, the intersection number of the loop with $Z$ is equal to $k$ mod $2$.  Thus we would like the intersection number with $Z_1$ to be $k$ mod $2$ as well.

We construct $Z_1$ as follows.  Let $\{\widehat{B} + i\}_{i \in \mathbb{Z}}$ denote the translates of $\widehat{B}$ under the $\mathbb{Z}$--action on $\widehat{M}$.  Let $N$ denote their union, the preimage of $B$ in $\widehat{M}$:
\[N = \bigcup_{i \in \mathbb{Z}} (\widehat{B} + i).\]
We set
\[N_+ = \bigcup_{i \geq 0} (\widehat{B} + i),\]
and let $\widehat{Z_1}$ denote the boundary of $N_+$ with respect to $N$.  Let $Z_1$ be the image of $\widehat{Z_1}$ in $B$, taken with multiplicity mod $2$ so that a point $q$ of $B$ is in $Z_1$ if and only if an odd number of preimages of $q$ are in $\widehat{Z_1}$.

The intersection number of $\gamma$ with $Z_1$ is equal to the sum of the intersection numbers of the various lifts of $\gamma$ with $\widehat{Z_1}$.  Given a lift of $\gamma$ stretching from $\widehat{p} + \ell$ to $\widehat{p} + \ell + k$, its intersection number with $\widehat{Z_1}$ is odd if and only if one of the endpoints $\widehat{p} + \ell$ and $\widehat{p} + \ell + k$ is in $N_+$ and the other is not.  Suppose we color in black all integers $i$ such that $\widehat{p} + i$ is not in $N_+$, and color in white all integers $i$ such that $\widehat{p} + i$ is in $N_+$.  We claim that for any such coloring of the integers, with all sufficiently negative integers black and all sufficiently positive integers white, the number of pairs $\ell, \ell + k$ that are different colors is congruent to $k$ mod $2$.  This statement is clearly true for $k = 1$, and for $k > 1$ we partition the integers into the $k$ congruence classes mod $k$, and apply the $k = 1$ statement to each congruence class.  Adding up $k$ odd numbers, the sum is congruent to $k$ mod $2$.  Thus the intersection number of $\gamma$ with $Z_1$ is indeed $k$ mod $2$.  Using $Z_1$ we can replace $Z \cap B$ by a subset of $S$, thus implying the desired inequality.
\end{proof}

Here is the lemma about perturbing the distance function, used in the proof above.

\begin{lem}\label{lem-transverse}
 Let $\widehat{M}$ be a covering space of a closed Riemannian manifold $M$, let $Z$ be a smooth submanifold of $M$, and let $\widehat{p}$ be any point in $\widehat{M}$.  Then there exists a smooth function $f \co \widehat{M} \rightarrow \mathbb{R}$, arbitrarily close to the function $\dist(-, \widehat{p})$, such that for all $t$ in a full-measure subset of the interval $[1/2, 1]$, the level set $f^{-1}(t)$ is a smooth submanifold of $\widehat{M}$, for which the image in $M$ is immersed with transverse self-intersections and transverse intersection with $Z$.
\end{lem}

\begin{proof}
 This lemma is a direct application of the Multijet Transversality
 Theorem as given in \cite[Chapter II, Theorem 4.13]{Golubitsky74}.  Here is the statement of the theorem in the case we need here.  For any natural number $s$, let $\widehat{M}^{(s)}$ denote the space
\[\widehat{M}^{(s)} = \{(x_1, \ldots, x_s) \in \widehat{M}^s : x_i \neq x_j\},\]
and for any function $f \co \widehat{M} \rightarrow \mathbb{R}$, its $s$--fold $0$--jet $j_s^0f$ is the function
\[j_s^0f \co \widehat{M}^{(s)} \rightarrow \widehat{M}^{(s)} \times \mathbb{R}^s\]
defined by
\[j_s^0f(x_1, \ldots, x_s) = (x_1, \ldots, x_s, f(x_1), \ldots, f(x_s)).\]
The theorem states that for any submanifold $W$ of $\widehat{M}^{(s)} \times \mathbb{R}^s$, the set of all $f \in C^\infty(\widehat{M}, \mathbb{R})$ such that $j_s^0f$ is transverse to $W$ is a comeager set.  (Comeager sets are dense, and any countable intersection of them is also dense.)  We apply the theorem repeatedly with various choices of $s$ and $W$, and select $f$ to satisfy all the transversality properties at once.

Let $G$ be the group of deck transformations of the covering space; in the modified stability lemma the group is $\mathbb{Z}$, so we use additive notation.  We want to check that wherever the $G$--translates of $f^{-1}(t)$ intersect, they are transverse.  If we suppose that $f$ is within $1$ of $\dist(-, \widehat{p})$ and $t \leq 1$, there are only finitely many translates of $f^{-1}(t)$ that could possibly intersect $f^{-1}(t)$; they are the translates by $k \in G$ such that $B(\widehat{p}, 2) + k$ intersects $B(\widehat{p}, 2)$.  Let $E$ be this finite set of translations $k$, including the identity.

Let $S$ be an arbitrary subset of $E$, with elements $k_1, \ldots, k_s$.  We set $W$ to be the set
\[W = \{(x + k_1, \ldots, x + k_s, t, \ldots, t) : x \in \widehat{M}, t \in \mathbb{R}\} \subseteq \widehat{M}^{(s)} \times \mathbb{R}^s.\]
If $j_s^0f$ is transverse to $W$ and $t$ is a regular value of the restriction of $f$ to the preimage of $W$, then the translates of $f^{-1}(t)$ by $-k_1, \ldots, -k_s$ are submanifolds transverse to each other.  Similarly, if we set 
\[W = \{x + k_1, \ldots, x + k_s, t, \ldots, t) : x \in \widehat{M} \text{ lies above } Z,\  t \in \mathbb{R}\},\]
then if $j_s^0f$ is transverse to $W$ and $t$ is a regular value, then the same translates of $f^{-1}(t)$ are (together) transverse to the preimage of $Z$ in $\widehat{M}$.

Applying the theorem for each subset of $E$, with and without $Z$, we obtain a dense set of functions $f$ for which a full-measure set of values $t$ satisfy the desired transversality conditions.
\end{proof}

\section{Proof of the modified volume theorem}\label{sect:proof}

In this section we show where to modify Guth's proof of the volume theorem (Theorem~\ref{thm-vol}) to get the modified volume theorem (Lemma~\ref{modified-volume}).  The proof of the volume theorem requires many steps, and we do not repeat them here.  We only point out the small difference.

Here we repeat the statements of the volume theorem and the modified volume theorem.  The volume theorem is about a different metric on a hyperbolic manifold.

\begin{repthm}{thm-vol}[Volume theorem]
 Let $(M, \hyp)$ be an $n$--dimensional closed hyperbolic manifold and let $g$ be another metric on $M$. Suppose that
\[V_{(\widetilde{M}, \widetilde{g})}(1) \leq V_{\mathbb{H}^n}(1),\]
where $\mathbb{H}^n$ is the $n$--dimensional hyperbolic
 space.  Then we have
\[\Vol(M, g) \geq \const(n) \cdot \Vol(M, \hyp).\]
\end{repthm}

The modified volume theorem is about a different manifold mapping to a hyperbolic manifold.

 \begin{replem}{modified-volume}[Modified volume theorem]
 Let $\Sigma$ be an $n$--dimensional closed hyperbolic manifold, and let $(Z, g)$ be an $n$--dimensional closed Riemannian manifold.  Suppose that $f \co (Z, g) \rightarrow \Sigma$ is a degree $1$ map such that $f(\gamma)$ is contractible for every loop $\gamma \subset Z$ of length less than $1$, and suppose that
\[V_{(Z, g)}(1/2) \leq V_0.\]
Then we have
\[\Vol (Z, g) \geq \const(n, V_0) \cdot \Vol(\Sigma, \hyp).\]
\end{replem}

The modified volume theorem is a generalization of the volume theorem. In fact, suppose that $M$ is a closed hyperbolic manifold
satisfying the hypothesis of the volume theorem, and we want to apply the modified volume theorem to obtain the conclusion of the volume theorem. As in
\cite[page 71]{Guth11} we take a finite cover of $M$ so that every loop
of length less than $2$ is contractible as follows. There
are only finitely many homotopy classes of loop with lengths less than
$2$. Since $M$ admits a hyperbolic metric, $\pi_1(M)$ is
residually finite. Hence these homotopy classes can be excluded after
taking a finite cover of $M$. We can scale the cover by $1/2$ so that the identity map of the cover
satisfies the hypothesis of the modified volume theorem.  We therefore get the
conclusion of the volume theorem for the cover. Since the volume
of the cover is the volume of $M$ times the degree of the cover, we
obtain the volume theorem.

In \cite{Guth11}, Guth proves the volume theorem from two lemmas: his Lemma~7 and his Lemma~9.  We modify his Lemma~7 to get our Lemma~\ref{lem:nerve}, but there is no change to Lemma~9 and essentially no change to how the two lemmas combine to prove the theorem.

To explain the setting for Lemma~7 of~\cite{Guth11},  we explain the rectangular nerve associated with a closed Riemannian manifold and its finite cover of balls. The rectangular
nerve was introduced by Guth in \cite{Guth11} and plays an important role in
the proof of Lemma \ref{modified-volume}.

Let $(Z,g)$ be a closed Riemannian manifold. Suppose that $B_i$,
$i=1,2,\cdots, D$ are closed balls in $Z$ such that the concentric balls
$\frac{1}{2}B_i$ cover $Z$ (here $\frac{1}{2}B_i$ is the ball with the
same center and half the radius of $B_i$). For such a cover $\{ B_i\}$ we define the
associated \emph{rectangular nerve} $N$ as follows. The nerve $N$ is a closed subcomplex of the rectangle $\prod_{i=1}^D [\, 0,r_i\,]$, where $r_i$ is
the radius of the ball $B_i$. Denote by $\phi_i$, $i=1,2,\cdots,D$ the coordinate function
of the rectangle. An open face $F$ of the rectangle is determined by
dividing the set $\{1,2,\cdots, D \}$ into three categories $I_0$,
$I_1$, and $I_{(\, 0,1\,)}$. The face $F$ is given by the equalities and
inequalities $\phi_i=0$ for $i\in I_0$, $\phi_i=r_i$ for $i\in I_1$, and
$0<\phi_i<r_i$ for $i\in I_{(\,0,1\,)}$. We denote $I_+:=I_1\cup I_{(\,0,1 \,)}$. A face $F$ is contained in the
nerve $N$ iff $\bigcap_{i\in I_+ (F)}B_i \neq \emptyset$ and if
$I_1(F)\neq \emptyset$.

We can define a natural map from $Z$ to the rectangular nerve $N$, which we call
$\phi:Z\to N$. Each $i$-th component of $\phi$ is a map $\phi_i:Z\to
[\,0,r_i\,]$ supported on $B_i$. We set $\phi_i=r_i$ on $\frac{1}{2}B_i$
and $\phi_i$ decreases to zero as $x$ approaches $\partial B_i$. The
function $\phi_i$ is chosen so that it is piecewise smooth and the
Lipschitz constant is at most $2$. Taking $\phi_i$ as a component we
define $\phi:Z\to \prod_{i=1}^{D}[\,0,r_i\,]$. The image
$ \phi(Z) $ lies in the rectangular nerve $N$.

Throughout this section we fix a special finite cover of balls $\{
B_i\}$ of $Z$, called a \emph{good
cover}, and consider the associated rectangular nerve $N$ and the natural map $\phi:Z\to N$.
We omit the definition of a good cover because we do not use the
definition throughout this paper. Refer to \cite{Guth11} for the precise
definition. Any ball in a good cover has radius at most $1/100$. 

Here is the original statement of Lemma~7 from~\cite{Guth11}.

\begin{lem}[{\cite[Lemma~7]{Guth11}}]\label{lem:original7}
Let $(M, g)$ be a closed aspherical manifold such that every loop of length at most $1$ is null-homotopic, and let $\phi \co M \rightarrow N$ be the map to the rectangular nerve $N$ constructed as above from a good cover of $M$.  Then there is a map $\psi \co N \rightarrow M$ so that the composition $\psi \circ \phi \co M \rightarrow M$ is homotopic to the identity.
\end{lem}

In the modified lemma, we take the good cover and rectangular nerve on one manifold $Z$, but the aspherical space without short nontrivial loops is another manifold $\Sigma$.  The proof is essentially the same.

\begin{lem}[Modified Lemma~7]\label{lem:nerve}  Let $(Z,g)$ be a closed Riemannian manifold, and let $\phi \co Z \rightarrow N$ be the map to the rectangular nerve $N$ constructed as above from a good cover of $Z$.  Let $\Sigma$ be a closed aspherical manifold, and suppose that a
 continuous map $f \co
 Z \rightarrow \Sigma$ has the property that for every loop $\gamma$ in $Z$
 of length at most $1$, the image $f(\gamma)$ in $\Sigma$ is null-homotopic.
 Then there is a map $\psi \co N \rightarrow \Sigma$ such that the
 composition $\psi \circ \phi \co Z \rightarrow \Sigma$ is homotopic to $f$.
 \begin{proof}The proof is a slight modification of
  \cite[Lemma 7]{Guth11} which is based on Gromov's argument
  in \cite{Gromov82}. We repeat the argument here for the completeness of
  this paper.

  We slightly homotope $\phi$ to get $\phi'$ which is
  simplicial with respect to some fine triangulations of $Z$ and $N$. We take
  the triangulation of $N$ so that it is a subdivision of the faces of $N$. 

  To construct $\psi$ we construct a $1$-skeleton map $\psi'$ from $N$
  to $Z$ just as in \cite{Guth11} and then compose it with $f$ to get a
  $1$-skeleton map from $N$ to $\Sigma$. We then extend the map $f\circ \psi'$ one
  skeleton at a time by using the hypothesis of the lemma and the
  resulting map will be $\psi$. Given any vertex $v\in N$ of the
  fine triangulation we take the
  smallest (open) face $F$ of $N$ that
  contains $v$. We pick an index in $I_+(F)$, say $i$, and define
  $\psi'(v)$ as the center of the ball $B_i$. Next let $e$ be a $1$-simplex
  in the triangulation of $N$. Since its endpoints lie in the same
  closed face of $N$ these endpoints are mapped by $\psi'$ to the centers of two overlapping balls of our
  good cover. Since the radii of the balls are at most $1/100$ the
  distance between these centers is at most $2/100$. We define $\psi'$
 to map $e$ to a
  curve in $Z$ joining the centers whose length is at most $2/100$.
  Now the boundary of each $2$-simplex of the triangulation of $N$ is
  mapped to a loop of length at most $6/100$ and the image of the
  loop under $f$ is contractible in $\Sigma$ from the hypothesis of $f$. Hence we can extend $f\circ
  \psi'$ to each $2$-simplex. Since $\Sigma$ is aspherical we can then extend
  the map to each higher-dimensional simplex. Now the map $\psi$ is constructed.

Similarly, to construct the homotopy $H$ between $f$ and $\psi\circ
  \phi$, we construct a $1$-skeleton map $H'$ to $Z$ just as in
  \cite{Guth11} and then compose with $f$ to get a $1$-skeleton map to
  $\Sigma$ that can be filled in to get $H$. For any vertex $v$ in the
  triangulation of $Z$ we put
  $H'(v,0):= \psi'\circ \phi'(v)$ and $H'(v,1):=v$. Recall that
  $\phi'(v)$ is a vertex in the triangulation of $N$ which is very close
  to $\phi(v)$.
  If $F$ is the
  smallest open face of $N$ containing $\phi'(v)$ then $\phi(v)$ may not be in
  $F$ but at least $\phi(v)$ is in a face bordering $F$. Hence $\psi'
  \circ \phi' (v)$ is the center of some ball $B_i$
  overlapping the other ball $B_j$ containing $v$. Thus $\psi'\circ
  \phi'(v)$ and $v$ can be connected by a curve of length at most
  $2/100$, and we define $H'$ on $v\times (\,0,1\, )$ by mapping the
  interval to the curve. Now let $e$ be a $1$-simplex in the
  triangulation of $Z$. Since our triangulation of $Z$ is fine we may
  assume that the length of $e$ is at most $1/100$. The image $\phi'(e)$ is either a point or a
  $1$-simplex of the triangulation of $N$. Hence the image $\psi'\circ
  \phi'(e)$ is either a point in $Z$ or a curve of length at most
  $2/100$. We have defined $H'$ on the boundary of $e\times
  (\,0,1\,)$. The image of the boundary under $H'$ is a loop in $Z$ of
  length at most $7/100$. Since the image of this loop under $f$ is
  contractible in $\Sigma$ we can extend $f \circ H'$ to $e\times (\,0,1\,)$
  for all $1$-skeleton $e$ and denote this extension as $H$. Since $\Sigma$
  is aspherical we can extend $H$ to $\Sigma \times (0,1)$ for each
  higher-dimensional simplex $\Sigma$. This $H$ is a homotopy between
  $\psi\circ \phi'$ and $f$. Since $\phi'$ is homotopic to
  $\phi$, $\psi \circ \phi$ is homotopic to $f$. This completes the proof.
  \end{proof}
 \end{lem}

In~\cite{Guth11} Guth finishes proving the volume theorem (Theorem~\ref{thm-vol}) from Lemma~7 and Lemma~9 of that paper.  In the remainder of this section we repeat the same steps to prove the modified volume theorem (Lemma~\ref{modified-volume}), with Lemma~7 of~\cite{Guth11} replaced by our version, Lemma~\ref{lem:nerve}.  To state Lemma~9 of~\cite{Guth11} and prove the modified volume theorem we recall Gromov's simplicial norm, which
was introduced in \cite{Gromov82}.
The simplicial norm is defined as follows. For every singular chain
$c$ on $Z$, the \emph{simplicial $\ell_1$ norm} of $c$, denoted $\|c\|$, is the sum of absolute
values of the (real) coefficients. For every real homology class
$h$, the \emph{simplicial norm} of $h$, denoted $\|h \|$, is the infimum of
$\| c \|$ over all cycles $c$ representing $h$.

Two facts about the simplicial norm which will be used in the proof of
Lemma \ref{modified-volume} are the following:
\begin{itemize}
 \item(Contractivity under mappings) For any continuous map $\phi:\Sigma\to \Sigma'$
      between two spaces, $\|\phi_{\ast}(h)\|\leq \|h\|$.
 \item(Gromov--Thurston's theorem, see \cite{Gromov82}) If $(\Sigma,\hyp)$ is a closed $n$-dimensional
      hyperbolic manifold then the simplicial volume of $\Sigma$ is
      equal to $\const(n) \Vol(\Sigma,\hyp)$ for a certain constant $\const(n)$.
 \end{itemize}The first item implies Euclidean spheres and tori have
      zero simplicial volume, because they have self-maps of nonzero degree. From Gromov--Thurston's theorem we can associate the simplicial volume of
a hyperbolic closed surface $\Sigma_g$ with its genus. In fact it is
known that $\| \Sigma_g \|=-2\chi(\Sigma_g)$. We will use these
computations in the appendix.

Here is Lemma~9 from~\cite{Guth11}.

 \begin{lem}[{\cite[Lemma 9]{Guth11}}]\label{lem:lem9 in larry}For each dimension $n$ and each number $V_0>0$, there is a
  constant $C(n,V_0)$ so that the following statement holds. If an
  $n$-dimensional closed Riemannian manifold $(Z,g)$ satisfies $V_{(Z,g)}(1/2)<V_0$, then we have $\|\phi_{\ast}([Z]) \|\leq C(n,V_0)\Vol(Z,g)$, where $\phi \co Z \rightarrow N$ is the map to the rectangular nerve $N$ associated to a good cover of $Z$. 
  \end{lem}

Using this lemma, we can finish the proof.

\begin{proof}[Proof of Lemma \ref{modified-volume}]
We choose a good cover on $Z$ and let $\phi \co Z \rightarrow N$ be the map to the rectangular nerve $N$.  Lemma~\ref{lem:lem9 in larry} says that we have
\[\|\phi_{\ast}([Z]) \|\leq C(n,V_0)\Vol(Z,g).\]
Then because we have assumed that the map $f \co Z \rightarrow \Sigma$ has degree $1$, Lemma~\ref{lem:nerve} gives
\[\psi_*(\phi_*([Z])) =  \deg f \cdot [\Sigma] = [\Sigma],\]
so contractivity of simplicial norm under mappings gives
\[\Vert [\Sigma] \Vert = \Vert \psi_*(\phi_*([Z])) \Vert \leq \Vert \phi_*([Z]) \Vert,\]
and Gromov--Thurston theorem gives
\[\Vol(\Sigma, \hyp) \leq \const(n) \cdot \Vert [\Sigma] \Vert.\]
Stringing the inequalities together we get
\[\Vol(\Sigma, \hyp) \leq \const(n, V_0) \cdot \Vol(Z, g),\]
completing the proof.
\end{proof}

\section{Codimension greater than one}\label{sect:codim 2}

In this section we prove Proposition~\ref{prop:codim2}, which says that the analogue of the area theorem (Theorem~\ref{thm-homology}) is false if we replace the $\mathbb{S}^1$ factor by a manifold of dimension greater than $1$.  This construction was found in conversation with Larry Guth.  We begin by repeating the statement of the proposition.

\begin{repprop}{prop:codim2}
Let $\Sigma$ be a closed hyperbolic manifold of dimension $n$, and let $X$ be any closed manifold of dimension $m \geq 2$.  There exists a constant $V_0 = 2\cdot V_{\mathbb{R}^{n+m}}(1)$ such that the following is true.  For any $\varepsilon > 0$ there is a Riemannian metric $g$ on $\Sigma \times X$ such that the volumes of unit balls in the universal cover satisfy
\[V_{(\widetilde{\Sigma \times X}, \widetilde{g})}(1) < V_0,\]
but simultaneously there is a submanifold isotopic to $\Sigma \times \ast$ with $n$--dimensional volume less than $\varepsilon$.
\end{repprop}

\begin{proof}
We think of 
$\Sigma \times X$ as a bundle over $X$ with fibers equal to $\Sigma$.  The metric we construct has the form $\lambda(x) \hyp \times g_X$ for some $\lambda \co X \rightarrow \mathbb{R}_{> 0}$, meaning that we choose a Riemannian metric $g_X$ on the base $X$ and then the fiber over each $x \in X$ is $\Sigma$ with its hyperbolic metric scaled by $\lambda(x)$.  If $\lambda$ is close to constant and all short loops in $X$ are null-homotopic, then the volume of a unit ball in $\widetilde{\Sigma \times X}$ is approximately equal to the product of the volumes of the unit ball in $\lambda$--scaled hyperbolic space and the unit ball in $X$.

The construction goes as follows.  First we scale an arbitrary Riemannian metric on $X$ so that $X$ is very big and flat, and all loops in $X$ of length at most $1$ are null-homotopic.  Then we replace a ball of radius $2$ in $X$ by the graph of an approximation to the delta function, something tall and skinny which we call the finger.  We choose $\lambda_{\min} > 0$ small enough that $(\Sigma, \lambda_{\min} \cdot \hyp)$ has volume less than $\varepsilon$, and set $\lambda(x) = \lambda_{\min}$ for $x$ at the tip of the finger.  We choose $\lambda_{\max}$ large and set $\lambda(x) = \lambda_{\max}$ for all $x$ outside the finger, so that the metric is close to Euclidean outside the finger.  Then along the length of the finger we change $\lambda$ very slowly between $\lambda_{\max}$ and $\lambda_{\min}$ as a function of the height.

If the finger is skinny enough, then the ball centered at the tip of the finger is not too large in volume.  (Here is where we use the assumption $m \geq 2$.)  And, if the finger is tall enough and $\lambda$ changes gradually enough (specifically, we want to bound the Lipschitz constant of $\log \lambda$), then the balls centered along the length of the finger are not too large.
\end{proof}

{\it Acknowledgments.} The authors would like to express our appreciation to Professor
Larry Guth for his guidance, valuable suggestions, and fruitful
conversations during the preparation of this paper. He introduced us to his papers~\cite{Guth11} and~\cite{Guth10} and suggested the problem of proving the area theorem (Theorem~\ref{thm-homology}).  He kindly allowed us to add Proposition \ref{prop:codim2} in this paper and encouraged us which we also would like to thank. Without these contributions this work would have never been completed.  We also thank the referee for a friendly and prompt report and a useful expository suggestion.

\bibliography{reference}{}
\bibliographystyle{amsalpha}

\section{Appendix}

We say that a closed minimal submanifold in a Riemannian manifold is
\emph{stable} if the second derivative of the area functional is nonnegative for
all variations. If the submanifold has codimension one and has a trivial
normal bundle, then the condition of stability can be simplified
into one of functions on the submanifold. 
 \begin{prop}[{The stability inequality}]Let $M$ be a $3$-dimensional
  Riemannian manifold. Suppose that $\Sigma\subseteq M$ is a stable minimal
  hypersurface with trivial normal bundle. Then
  \begin{align*}
   \int_{\Sigma}  | \nabla \phi |^2 + K_{\Sigma} \phi^2 -\frac{1}{2}\scal_g \phi^2 \geq 0 
  \end{align*}for any $\phi \in C^{\infty}_0(\Sigma)$, where
  $K_{\Sigma}$ is the Gauss curvature of $\Sigma$.
  \begin{proof}This follows from the combination of (1.153) and (1.159) in \cite{Colding11}.
   \end{proof}
  \end{prop}
 \begin{proof}[Proof of the Kronheimer--Mrowka theorem]
  By \cite[Lemma 2.1]{Hass88} there is a finite union of oriented
  closed surfaces (with multiplicities) which minimizes area among
  surfaces in $[\Sigma]$. We denote it as $n_1 \Sigma_1 + n_2 \Sigma_2
  +\cdots n_k \Sigma_k$. Then we can apply the stability inequality and get
  \begin{align*}
   \sum_{i=1}^k |n_i| \int_{\Sigma_i} |\nabla \phi|^2 + K_{\Sigma_i} \phi^2 -\frac{1}{2}
   \scal_g \phi^2 \geq 0 ,
   \end{align*}for any $\phi \in C^{\infty}(\bigcup_{i=1}^k \Sigma_i
  )$. Taking $\phi \equiv 1$ we get 
  \begin{align*}
   \sum_{i=1}^k |n_i|\int_{\Sigma_i} K_{\Sigma_i}\geq
   \sum_{i=1}^k \frac{|n_i|}{2}\int_{\Sigma_i}\scal_g\geq \sum_{i=1}^k\frac{|n_i|}{2}\int_{\Sigma_i} \scal_{g_0}.
   \end{align*}Since
  $\scal_{g_0}=\scal_{(\Sigma,\hyp)}+\scal_{(\mathbb{S}^1,g_{\mathbb{S}^1})}=-2$ the
  right-hand side of the above inequality is $-\Area(n_1\Sigma_1+\cdots
  +n_k \Sigma_k)$.

  Now we want to estimate the left-hand side.
  Observe first that each $\Sigma_i$ is not a 2-sphere. This is because
  $\Sigma\times \mathbb{S}^1$ is aspherical and any $2$-spheres are
  null-homotopic, thus null-homologous. Then we have
  $\|[\Sigma_i]\|=-2\chi (\Sigma_i)$ for each $i$. We consider the composition of the
  inclusion of $\sum_{i=1}^k n_k \Sigma_i$ into
  $\Sigma\times \mathbb{S}^1$ with the projection from $\Sigma\times
  \mathbb{S}^1$ to $\Sigma$, and call this composition $f$. Since
  $\sum_{i=1}^k n_i \Sigma_i$ is homologous to $\Sigma$, $f$ has odd degree, in particular
  nonzero degree. Using simplicial norm we have
  \begin{align}\label{1appendix}
   |\deg f| \cdot \| [\Sigma]\|= \|f_{\ast}\sum_{i=1}^k n_i
   [\Sigma_i]\|\leq \sum_{i=1}^k |n_i|
   \|[\Sigma_i]\|=-2\sum_{i=1}^k|n_i|\chi (\Sigma_i).
   \end{align}By the Gromov--Thurston theorem, we get
  \begin{align}\label{2appendix}
   \pi \| [\Sigma]\|=\Vol (\Sigma, \hyp).
   \end{align}Combining the above two inequalities (\ref{1appendix}) and
  (\ref{2appendix}) (with the Gauss-Bonnet theorem) implies
  \begin{align*}
  \sum_{i=1}^k |n_i|\int_{\Sigma_i} K_i \leq -\Vol(\Sigma,\hyp)
   \end{align*}This implies the desired conclusion.
 \end{proof}
\end{document}